\documentclass[11pt]{article}

\usepackage{amsmath}
\usepackage{amsfonts}
\usepackage{amsthm}
\usepackage{amssymb}
\usepackage{MnSymbol}
\usepackage{bbm}
\usepackage{listings}
\usepackage{enumerate}
\usepackage{cite}
\usepackage{etoolbox}
\usepackage{ifthen}

\newcommand{\R}{\mathbb{R}}

\newcommand{\E}{\mathbb{E}}

\newcommand{\calH}{\mathcal{H}}

\newcommand{\inr}[2]{\langle #1, #2 \rangle}

\newcommand{\normal}{\mathcal{N}}

\newcommand{\pdiff}[2]{\frac{\partial #1}{\partial #2}}
\newcommand{\pdiffII}[3]{\ifstrequal{#2}{#3}
{\frac{\partial^2 #1}{\partial #2^2}}
{\frac{\partial^2 #1}{\partial #2 \partial #3}}
}
\newcommand{\diffII}[3]{\ifthenelse{\equal{#2}{#3}}
{\frac{d^2 #1}{d #2^2}}
{\frac{d^2 #1}{d #2 d #3}}
}
\newcommand{\diff}[2]{\frac{d #1}{d #2}}

\newcommand{\grad}{\nabla}

\newcommand{\eps}{\epsilon}

\let\Pr\relax
\DeclareMathOperator{\Pr}{Pr}

\DeclareMathOperator{\spn}{span}

\DeclareMathOperator{\sgn}{sgn}
\DeclareMathOperator{\Var}{Var}
\DeclareMathOperator{\Cov}{Cov}

\DeclareMathOperator{\Inf}{Inf}

\newtheorem{theorem}{Theorem}[section]

\newtheorem{lemma}[theorem]{Lemma}

\newtheorem{proposition}[theorem]{Proposition}

\newtheoremstyle{example}{\topsep}{\topsep}%
     {}%         Body font
     {}%         Indent amount (empty = no indent, \parindent = para indent)
     {\bfseries}% Thm head font
     {}%        Punctuation after thm head
     {\newline}%     Space after thm head (\newline = linebreak)
     {\thmname{#1}\thmnumber{ #2}\thmnote{ #3}}%         Thm head spec

\theoremstyle{example}

\newcommand{\replace}{\oslash}

\DeclareMathOperator{\noisestab}{NS}

\begin{document}

\title{Noise Stability and Correlation with Half Spaces}
\author{
	Elchanan Mossel
	\thanks{University of California, Berkeley and University of Pennsylvania; \texttt{mossel@wharton.upenn.edu}}
	\and
	Joe Neeman
	\thanks{UT Austin and University of Bonn; \texttt{joeneeman@gmail.com}}
}

\maketitle 

\begin{abstract}
Benjamini, Kalai and Schramm showed that a monotone function $f  : \{-1,1\}^n
\to \{-1,1\}$ is noise stable if and only if it is correlated with a half-space
(a set of the form $\{x: \inr{x}{a} \le b\}$).
 
We study noise stability in terms of correlation with half-spaces for general
(not necessarily monotone) functions.  We show that a function $f: \{-1, 1\}^n
\to \{-1, 1\}$ is noise stable if and only if it becomes correlated with a
half-space when we modify $f$ by randomly restricting a constant fraction of
its coordinates.

Looking at random restrictions is necessary: we construct noise stable
functions whose correlation with any half-space is $o(1)$. The examples further
satisfy that different restrictions are correlated with different half-spaces:
for any fixed half-space, the probability that a random restriction is
correlated with it goes to zero.

We also provide quantitative versions of the above statements, and versions
that apply for the Gaussian measure on $\R^n$ instead of the discrete cube.
Our work is motivated by questions in learning theory and a recent question
of Khot and Moshkovitz. 
\end{abstract}

\newpage

\section{Introduction}
 
In a seminal paper, Benjamini, Kalai and Schramm~\cite{BeKaSc:99} related noise stability to correlation with half-spaces by showing that a monotone boolean function is noise stable if and only if it is correlated with a half-space. Our interest in this paper is relating noise stability with correlation with half-spaces for general boolean functions. Our results are motivated by recent work of Khot and Moshkovitz whose goal is to construct a Lasserre 
integrality gap for the Unique Games problems as well as by natural problems in learning theory. 
 
In the following subsections we introduce the setup and results in the boolean and Gaussian cases and discuss the motivation for our work.

\subsection{The boolean setting}\label{sec:intro-boolean}

Let $\mu_n$ denote the uniform measure on $\{-1, 1\}^n$. For $t \ge 0$, let $P_t$
denote the Bonami-Beckner semigroup, defined by
\[
    (P_t f)(x) = (P_{t,1} f)(x) = \E f + e^{-t} (f(x) - \E f)
\]
in the case $n = 1$ and $P_{t,n} = P_{t,1}^{\otimes n}$ otherwise. The \emph{boolean
noise stability} of a set $A \subset \{-1, 1\}^n$ is
\[
    \noisestab_t(A) = \E [1_A P_t 1_A],
\]
where the expectation is taken with respect to $\mu_n$. 
Since $P_{t} = P_{t/2} P_{t/2}$ and $P_t$ is
self-adjoint, we may also write $\noisestab_{t}(A) = \E [(P_{t/2} 1_A)^2]$. Then
\[
    \noisestab_t(A) - \mu_n(A)^2 = \noisestab_t(A) -  (\E P_{t/2} 1_A)^2 = \Var(P_{t/2} 1_A) \ge 0;
\]
the quantity $\Var(P_t 1_A)$ turns out to be a useful re-parametrization of the usual boolean noise sensitivity.

We say that a sequence $A_i: \{-1, 1\}^{n_i}$ of sets is \emph{noise sensitive} if
for every $t > 0$, $\Var(P_t 1_{A_i}) \to 0$ is $i \to \infty$. Otherwise, we say that the
sequence $A_i$ is \emph{noise stable}.

A \emph{half-space} is a set of the form $\{x \in \{-1, 1\}^n: \inr x a \le b\}$; write $\calH_n$ for the set
of all half-spaces in $\{-1, 1\}^n$. Define
\[
    M(A) = \sup_{A \in \calH_n} \Cov(1_A, 1_B).
\]
Clearly $0 \le M(A) \le \frac 14$ for all $A$.

The set $A \subset \{-1, 1\}^n$ is \emph{monotone} if whenever $x \in A$ and $y \ge x$ coordinatewise then $y \in A$.
Benjamini, Kalai, and Schramm~\cite{BeKaSc:99} proved that a sequence $A_i$ of monotone sets is
noise sensitive if and only if $M(A_i) \to 0$. In this article, we explore removing the condition of
monotonicity. First, we show that one direction of Benjamini et al.'s equivalence fails when the $A_i$
are allowed
to be non-monotone. In particular, we construct a sequence of sets $B_i \subset \{-1, 1\}^{n_i}$
such that $M(B_i) \to 0$ but $\noisestab_t(B_i) \not \to 0$; in other words, noise-stable sets
are not necessarily correlated with any half-spaces.

Although noise-stable sets may not be correlated with half-spaces, there is a characterization of
noise stability in terms of half-spaces; this characterization requires the notion of a
\emph{restriction}.
For $z \in \{-1, 0, 1\}$ and $y \in \{-1, 1\}$, define $z \replace y \in \{-1, 1\}$ by
\[
    z \replace y = \begin{cases}
      y & \text{if $z = 0$} \\
      z & \text{otherwise.}
    \end{cases}
\]
For $z \in \{-1, 0, 1\}^n$ and $y \in \{-1, 1\}^n$, define $z \replace y \in \{-1, 1\}^n$
coordinatewise: $(z \replace y)_i = z_i \replace y_i$.
For a set $B \subset \{-1, 1\}^n$, define a restriction of $B$ by
\[
    B_z = \{x \in \{-1, 1\}^n: z \replace x \in B\}.
\]
Write $\mu_t$ for the measure on $\{-1, 0, 1\}^n$ under which each coordinate is independent, equal
to zero with probability $e^{-t}$, and chosen uniformly from $\{-1, 1\}$ otherwise.

Our main theorem, in its qualitative form (its analogous quantitative versions are Theorem~\ref{thm:boolean-restriction}
and Theorem~\ref{thm:boolean-converse}), says that a set is noise stable if and only if we can make it
correlated with a half-space by randomly restricting a constant fraction of its coordinates.

\begin{theorem}\label{thm:boolean-intro}
  The sequence $B^{(i)} \subset \{-1, 1\}^{n_i}$ is noise stable if and only if
  there are some $t, \epsilon > 0$ such that for all sufficiently large $i$,
  $M(B^{(i)}_{Z}) \ge \epsilon$ with probability at least $\epsilon$, where $Z
  \sim \mu_t$.
\end{theorem}

Since the notion of taking restrictions may seem artificial, it is natural
to ask whether taking restrictions in Theorem~\ref{thm:boolean-intro} is really necessary.
That is, could it be that $B^{(i)}$ noise stable already implies that $M(B^{(i)}) \not \to 0$?
In fact, this is not the case. As an example, take $n_m = n^2$ and consider the sets
$B^{(m)} \subset \{-1, 1\}^{n_m}$ defined by
\[
  B^{(m)} = \bigg\{x: \sum_{i=1}^m \bigg(\frac{1}{\sqrt m} \sum_{j=(i-1)m + 1}^{im} x_j \bigg)^2 \le m \bigg\}.
\]

\begin{proposition}
  The sets $B^{(m)}$ are noise stable, but $M(B^{(m)}) \le C m^{-1/200}$ for a universal constant $C$.
\end{proposition}

\subsection{The Gaussian setting}

The preceding results also make sense in a Gaussian setting:
Let $\gamma_n$ denote the standard Gaussian measure on $\R^n$ and
write $P_t$ for the Ornstein-Uhlenbeck semigroup, defined by
\[
    P_t f(x) = \E f(e^{-t} x + \sqrt{1-e^{-2t}} X), \qquad X \sim \gamma_n.
\]
(Here and elsewhere we will reuse symbols that we also used in the boolean setting; however,
the meaning should always be clear from the context.)
The \emph{Gaussian noise stability} of a set $A \subset \R^n$ is
\[
    \noisestab_t(A) = \E [1_A P_t 1_A].
\]
As in the boolean case, we have $\noisestab_t(A) - \gamma_n(A)^2 = \Var(P_{t/2} 1_A)$;
we say that a sequence $A_i$ of sets is \emph{noise sensitive} if
$\Var(P_t 1_{A_i}) \to 0$ for all $t > 0$, and we say that $A_i$ is \emph{noise stable} otherwise.
A \emph{half-space} is a set of the form $\{x \in \R^n: \inr x a \le b\}$; write $\calH_n$ for the set
of all half-spaces in $\R^n$ and define
\[
    M(A) = \sup_{A \in \calH_n} \Cov(1_A, 1_B).
\]

In the setting above, we prove that a sequence of sets is noise stable if and only if
by scaling and randomly shifting it, we make them correlated with half-spaces.
Specifically, given $B \subset \R^n$, $t \ge 0$, and $y \in \R^n$, define
\[
  B_{t,y} = \{x \in \R^n: \sqrt{1-e^{-2t}} c + e^{-t} y \in B\}.
\]

\begin{theorem}\label{thm:gaussian-intro}
  The sequence $B^{(i)} \subset \R^n$ is noise stable if and only if there are
  some $t, \epsilon > 0$ such that for all sufficiently large $i$,
  $M(B^{(i)}_{t,Y}) \ge \epsilon$ with probability at least $\epsilon$, where
  $Y \sim \gamma_n$.
\end{theorem}

As in the boolean case, one can find examples showing that Theorem~\ref{thm:gaussian-intro}
would be false if we didn't introduce the scaling and random shifting. In this
case, the example is very easy: let $B^{(n)} \subset \R^n$ be the Euclidean ball of radius
$\sqrt n$.

\begin{proposition}
  The sets $B^{(n)}$ are noise stable, but $M(B^{(n)}) \le n^{-1/2}$.
\end{proposition}

One can learn a little more from this example. First, note that any
restrictions of $B^{(n)}$ are also Euclidean balls. In the Gaussian setting,
therefore, unlike in the boolean one, noise stability does not imply that
random restrictions are correlated with half-spaces.  Another observation
(since $B^{(n)}$ is rotationally invariant) is that noise stable sets do not
necessarily ``encode'' directions. We make this more precise in
Proposition~\ref{prop:no-direction}, which says that even though random shifts
and scalings of $B^{(n)}$ are correlated with half-spaces, the directions
in which those half-spaces point are unpredictable.

\subsection{Motivation} 
Our work is motivated by extending the results of~\cite{BeKaSc:99} to non-monotone functions, as well by the following motivations:

\begin{itemize}
\item 
 In a recent work Khot and Moshkovitz~\cite{KhotMoshkovitz:16}, proposed a Lasserre integrality gap for the Unique Games problem. The proposed construction is based on the assumption that in a certain family of functions, the most stable functions are half-spaces. More specifically~\cite{KhotMoshkovitz:16} considers $f : \R^n \to \{-1,+1\}$ which satisfy
\[
f(-x) = f(x+e_i) = -f(x),
\]
for all $x$ and for the standard basis vectors $e_i$; they asked whether the
most stable functions in this family are of the form $\sgn(\sum_i \sigma_i
x_i)$ where $\sigma_i \in \{-1,1\}^n$, and also whether every function that is
almost as noise stable as possible must be correlated with a function of this
form.

In this context, it is natural to ask whether every noise stable function is
correlated with a half-space.  This is the question we address in this paper.
However, since our functions are not required to satisfy $f(x + e_i) = -f(x)$,
our results and examples do not have direct implications for the proposed
Lasserre integrality gap instances.

\item
It is well known that the class of functions having a constant fraction (resp. most) of their Fourier mass 
on ``low'' coefficients can be weakly (resp. strongly) learned under the uniform distribution~\cite{LiMaNi:93,KlOdSe:04}. 
In particular, noise stable functions can be weakly learned. 
On the other hand, the most classical learning algorithms involve learning half-spaces. 
Thus it is natural to ask if there is more direct relation between the weak learnability of noise stable functions and the learnability of half-spaces. Our examples seem to provide a negative answer to this question. 

\end{itemize}

\section{The Gaussian case}
For this section, let $X \sim \gamma_n$.
Recall that the Ornstein-Uhlenbeck semi-group is defined by
\[
P_t f(x) = \E f(e^{-t} x + \sqrt{1-e^{-2t}} X).
\]
For $t \in \R$ and $y \in \R^n$, define $f_{t,y}$ by
$f_{t,y}(x) = f(\sqrt{1-e^{-2t}}x + e^{-t} y)$.

\begin{theorem}\label{thm:gaussian}
  For any measurable $f: \R^n \to [0, 1]$ and any $t > 0$,
  \[
      \E M(f_{t,Y}) \ge c (e^{2t} - 1) \Var(P_t f),
  \]
  where $c > 0$ is a universal constant and $Y \sim \gamma_n$.
\end{theorem}

\subsection{An example}\label{sec:gaussian-example}
It is natural to ask whether one needs to replace $f$ by $f_{t,y}$ in order to find
a correlated half-space. Indeed, a simple example shows that $f$ itself may not
be correlated with a half-space: let $B_n \subset \R^n$ be the Euclidean ball of radius $\sqrt n$.
First, we note that for sufficiently small $t$, $\Var(P_t 1_{B_n})$ is bounded away from zero as
$n \to \infty$. (This is already well-known~\cite{Kane:11}, since $B_n$ is obtained by
thresholding a quadratic function, but the computation in our special case is quite easy.)

\begin{proposition}\label{prop:example-gaussian-stable}
  For any $n$ and any $t > 0$,
  \[
    \Var(P_t 1_{B_n}) \ge \frac 14 - \frac{\arccos(e^{-2t})}{\sqrt 2 \pi} - o_n(1).
  \]
  In particular $B_n$ is noise stable.
\end{proposition}

\begin{proof}
  For a set of $B$ of smooth boundary, we may define the \emph{Gaussian perimeter} of
  $B$ as
  \[
  \int_{\partial B} \diff{\gamma_n}{\lambda}(x)\, d\calH_{n-1}(x),
  \]
  where $\calH_{n-1}$ denotes the $(n-1)$-dimensional Hausdorff measure and
  $\diff{\gamma_n}{\lambda}$ denotes the Gaussian density with respect to the Lebesgue
  measure.
  Since the Gaussian density restricted to $\partial B_n$ takes the constant
  value $(2\pi e)^{-n/2}$ and the Euclidean surface area of
  $B_n$ is $\sqrt n^{n-1} \cdot 2 \pi^{n/2} / \Gamma(n/2)$, it follows that
  the Gaussian perimeter of $B_n$ is
  \[
  \frac{2 n^{n/2 - 1/2}}{(2e)^{n/2} \Gamma(n/2)} \sim \frac{1}{\sqrt{\pi}},
  \]
  where the approximation follows from Stirling's formula.

  On the other hand, Ledoux~\cite{Ledoux:94} proved that if $P$ is the Gaussian perimeter of $B$
  then
  \[
  \E [1_B(1_B - P_t 1_B)] \le \frac{\arccos(e^{-t}) P}{\sqrt{2\pi}}.
  \]
  Plugging in our asymptotics for the Gaussian perimeter of $B_n$, we have
  \[
  \E [1_{B_n}(1_{B_n} - P_t 1_{B_n})] \le (1 + o_n(1))\frac{\arccos(e^{-t})}{\sqrt{2} \pi}.
  \]
  Since $P_t = P_{t/2} P_{t/2}$ and $P_{t/2}$ is self-adjoint, this may be rearranged into
  \[
  \E [(P_{t/2} 1_{B_n})^2] \ge \Pr(B_n) - (1 + o_n(1)) \frac{\arccos(e^{-t})}{\sqrt{2} \pi}.
  \]
  Since $\Pr(B_n) = \frac 12 + o_n(1)$, this proves the claim.
\end{proof}

Next, we observe that $B_n$ is not correlated with any half-space:
\begin{proposition}\label{prop:example-gaussian-no-half-space}
  $M(B_n) \le n^{-1/2}$.
\end{proposition}

In particular, Propositions~\ref{prop:example-gaussian-stable}
and~\ref{prop:example-gaussian-no-half-space} together imply that Theorem~\ref{thm:gaussian}
would no longer be true if $f_{t,y}$ were replaced by $f$.

\begin{proof}
  Since $B_n$ is rotationally invariant, it suffices to consider half-spaces
  of the form $A_i := \{x: x_i \le b\}$. Since $\Pr(A_i) = \Phi(b)$,
  \[
  \Cov(1_{B_n}, 1_{A_i}) = \E [1_{B_n} (1_{A_i} - \Phi(b))].
  \]
  Now let $f_i = 1_{A_i} - \Phi(b)$. Then the $f_i$ are orthogonal and satisfy
  $\|f_i\|_2 \le 1$. Hence,
  \[
  1 \ge \|1_{B_n}\|_2^2 \ge \sum_{i=1}^n \E [1_{B_n} f_i]^2
  = n \E [1_{B_n} f_1]^2,
  \]
  and so $\E [1_{B_n} f_1] \le n^{-1/2}$.
\end{proof}

A very similar argument shows that even though shifts of $A_n$ may be correlated
with half-spaces, the half-spaces are pointed in unpredictable directions.

\begin{proposition}\label{prop:no-direction}
  Let $g = 1_{B_n}$ and let $g_{t,y}(x) = g(\sqrt{1-e^{-2t}} x + e^{-t} y)$.
  For any half-space $A$,
  \[
  \E_Y [\Cov(g_{t,Y}, 1_A)^2] \le \frac 1n.
  \]
  In particular, Chebyshev's inequality implies that for any $u > 0$,
  with probability at least $1 - u^{-2}$ over $Y \sim \normal(0, I_n)$
  \[
  |\Cov(g_{t,Y}, 1_A)| \le \frac{u}{n}.
  \]
\end{proposition}

\begin{proof}
  Let $A_i = \{x: x_i \le b\}$ and $f_i = 1_{A_i} - \Phi(b)$. As in the proof
  of the previous proposition, for any $Y$ and $t$,
  \[
  1 \ge \sum_{i=1}^n \E[g_{t,Y} f_i]^2 = n \E[g_{t,Y} f_1]^2
  = n \Cov(g_{t,Y}, f_1)^2.
  \]
  Taking the expectation over $Y$ completes the proof.
\end{proof}

\subsection{Proof of Theorem~\ref{thm:gaussian}}

For $f \in L_2(\gamma_n)$, define $w_1(f) = \sum_i \E[X_i f(X)]^2$.
Using the integration by parts formula $\E [X_i f(X)] = \E[\pdiff{f}{x_i}(X)]$, we may
also write
$w_1(f) = |\E \grad f|^2$.
The proof of Theorem~\ref{thm:gaussian} goes in two steps: first, we show that if
$w_1(f)$ is non-negligible then there exists a half-space correlated with $f$.
Then, we show that for a random $Y \sim \gamma_n$, $w_1(f_{t,y})$ is non-negligible in
expectation.

\begin{proposition}\label{prop:gaussian-w1-to-halfspace}
Take $f:  \R^n \to [0, 1]$. If $w_1(f) = \epsilon^2$ and $\Var(f) = \sigma^2$
then there exists a half-space $A \subset \R^n$ with
\[
 \Cov(f, 1_A) \ge \frac{\epsilon^2}{8\pi \sigma}.
\]
\end{proposition}

Before proving Proposition~\ref{prop:gaussian-w1-to-halfspace}, we will show that it suffices
to find a half-space correlated with $P_t f$:

\begin{lemma}\label{lem:gaussian-smoothed-cov-to-cov}
If there exists a half-space $A$ with
$\Cov(P_t f, 1_A) \ge \delta$ then there exists a half-space $A'$
with $\Cov(f, 1_{A'}) \ge \delta$.
\end{lemma}

\begin{proof}
Since $P_t$ is self-adjoint, we have
\[
 \Cov(P_t f, 1_A) = \Cov(f, P_t 1_A) = \E [(f - \E f) P_t 1_A].
\]
Assuming that $A \in \{x_1 \le b\}$, we can write
\begin{align*}
 (P_t 1_A)(x)
 &= \int_{\R^n} 1_A(e^{-t} x + \sqrt{1-e^{-2t}} y) \, d\gamma_n(y) \\
 &= \int_{\R^n} 1_{\{x \in A - \sqrt{e^{2t} - 1} y\}} \,d\gamma_n(y).
\end{align*}
In other words, if we set $A_y = A - \sqrt{e^{2t} - 1} y$
then we may write $P_t 1_A$ as an average of other half-spaces:
$P_t 1_A = \E_Y 1_{A_Y}$. Hence,
\[
 \E [(f - \E f) P_t 1_A]
 = \E [(f(X) - \E f) 1_{A_Y}(X)
\]
where $X$ and $Y$ are independent standard Gaussian vectors.
Then there exists some $y \in \R^n$ with
\[
 \Cov(f, 1_{A_y}) \ge \E [(f(X) - \E f) 1_{A_Y}(X)] = \Cov(P_t f, 1_B).
 \qedhere
\]
\end{proof}

\begin{proof}[Proof of Proposition~\ref{prop:gaussian-w1-to-halfspace}]
Write $f = \E f + f_1 + f_2$ where $f_1 \in \spn\{x_1, \dots, x_n\}$ and
$f_2$ is orthogonal to both $f_1$ and 1. We may assume by rotational
invariance that $f_1(x) = \epsilon x_1$. Let $A = \{x_1 \ge 0\}$.
Since $P_t f_1 = e^{-t} f_1$ and $\E[(P_t f_2)^2] \le e^{-4t} \E[f_2^2]$,
we have
\begin{align*}
\E [P_t f 1_A] - \E f \E 1_A &= \E[ 1_A P_t f_1] + \E[ 1_A P_t f_2] \\
 &= e^{-t} \frac{\epsilon}{\sqrt{2\pi}} + \E[1_A P_t f_2] \\
 &\ge e^{-t} \frac{\epsilon}{\sqrt{2\pi}} - e^{-2t} \|1_A\|_2 \|f_2\|_2 \\
 &\ge e^{-t} \frac{\epsilon}{\sqrt{2\pi}} - \frac{1}{\sqrt 2} e^{-2t} \sigma.
\end{align*}
Now take $t$ so that $2 \sqrt{\pi} e^{-t} \sigma = \epsilon$. Then
\[
\E[P_t f 1_A] - \E f \E 1_A \ge e^{-t} \frac{\epsilon}{2\sqrt {2\pi}}
 = \frac{\epsilon^2}{8\pi \sigma}.
\]
By Lemma~\ref{lem:gaussian-smoothed-cov-to-cov}, there exists
some half-space $A'$ with $\Cov(f, 1_{A'}) \ge \frac{\epsilon^2}{8\pi \sigma}$.
\end{proof}

The second step in the proof of Theorem~\ref{thm:gaussian} is to show that
if a function $f$ is noise stable then it has some shifts $f_{t,y}$ with
non-negligible $w_1(f_{t,y})$. In order to do this, recall the Gaussian
Poincar\'e inequality (see, e.g.~\cite{BaGeLe:14}), which states that
$\Var(f) \le \E |\grad f|^2$ for any $f$ with continuous derivatives.

\begin{proposition}\label{prop:exp-w1}
For any $f$ and any $t > 0$, if $Y \sim \normal(0, I_n)$ then
\[
    \E w_1(f_{t,Y}) \ge (e^{2t} - 1) \Var(P_t f).
\]
\end{proposition}

\begin{proof}
Since smooth functions are dense in $L_2(\gamma_n)$, and since both $w_1(f)$ and $\Var(P_t f)$
are preserved under $L_2(\gamma_n)$ convergence, we may assume that $f$
is smooth. Then $\grad f_{t,y} = \sqrt{1-e^{-2t}} (\grad f)_{t,y}$. Hence,
\[
 w_1(f_{t,y}) = |\E \grad f_{t,y}|^2 =
 (1 - e^{-2t}) |\E \grad f(\sqrt{1-e^{-2t}} X + e^{-t} y)|^2.
\]
Now set $Y$ to be a standard Gaussian vector in $\R^n$, independent of $X$.
Then
\begin{align*}
 \E w_1(f_{t,Y})
 &= (1 - e^{-2t}) \E_Y |\E_X \grad f(\sqrt{1-e^{-2t}} X + e^{-t} Y)|^2 \\
 &= (1 - e^{-2t}) \E_Y |(P_t \grad f)(Y)|^2 \\
 &= (e^{2t} - 1) \E_Y |(\grad P_t f)(Y)|^2,
\end{align*}
where the last line follows because $P_t \grad f = e^t \grad P_t f$.
Finally, the Poincar\'e inequality applied to $P_t f$ yields
\[
    \E w_1(f_{t,Y}) = (e^{2t} - 1) \E |\grad P_t f|^2 \ge (e^{2t} - 1) \Var(P_t f).
 \qedhere
\]
\end{proof}

\begin{proof}[Proof of Theorem~\ref{thm:gaussian}]
By Proposition~\ref{prop:exp-w1}, there exists some $y \in \R^n$ such that
$w_1(f_{t,y}) \ge \Var(P_t f)$.
Now, $f_{t,y}$ takes values in $[0, 1]$ and hence it has variance at
most 1. By Proposition~\ref{prop:gaussian-w1-to-halfspace}, there exists a
half-space $A$ with $\Cov(f_{t,y}, 1_A) \ge c (e^{2t} - 1) \Var(P_t f)$.
\end{proof}

\subsection{The converse of Theorem~\ref{thm:gaussian}}

The following result is a (qualitative) converse of Theorem~\ref{thm:gaussian}.
For example, it implies that if $M(f_{s,Y})$ is non-negligible with constant
probability then $f$ is noise stable.
In particular, together with Theorem~\ref{thm:gaussian} it implies Theorem~\ref{thm:gaussian-intro}.

\begin{theorem}\label{thm:gaussian-converse}
  For any $0 < r < s$ and any $f: \R^n \to [0, 1]$,
  \[
    (1 - e^{-2(s-r)}) \Var(P_r f) \ge 4 \E_Y M^2(f_{s,Y}) - C \left(\frac{1-e^{-2r}}{1-e^{-2s}}\right)^{1/4}.
  \]
\end{theorem}

\begin{lemma}\label{lem:half-space-l2}
  For any half-space $A$ and any $t > 0$,
  \[
    \E [ (1_A - P_t 1_A)^2 ] \le \frac 1\pi \arccos(e^{-t}).
  \]
\end{lemma}

\begin{proof}
  Ledoux's bound gives
  \[
    \E [(1 - 1_A) P_{t} 1_A] \le \frac{\arccos(e^{-t})}{2\pi}.
  \]
  Rearranging this,
  \begin{equation}\label{eq:ledoux-rearranged}
    \E [1_A P_t 1_A] \ge \gamma_n(A) - \frac{\arccos(e^{-t})}{2\pi}.
  \end{equation}
  On the other hand,
  \[
    \E [ (1_A - P_t 1_A)^2 ] = \gamma_n(A) - 2 \E[1_A P_t 1_A] + \E [(P_t 1_A)^2]
    \le 2 \gamma_n(A) - 2 \E [1_A P_t 1_A].
  \]
  Applying~\eqref{eq:ledoux-rearranged} completes the proof.
\end{proof}

Next, we show that any set which is correlated with a half-space must be noise stable (indeed,
almost as noise stable as the half-space itself).

\begin{proposition}\label{prop:correlation-to-stability}
  Suppose that $A \subset \R^n$ is a half-space. Then for any $f: \R^n \to [0, 1]$
  and any $t > 0$,
  \begin{align*}
    \Var(P_t f)
    &\ge \frac{\Cov(A, f)^2}{\big(\gamma_n(A) (1-\gamma_n(A))\big)^2} \Var(P_t 1_A) - \frac{\sqrt{\arccos(e^{-2t})}}{\sqrt \pi} \\
    &\ge 4 \Cov(A, f)^2 - C t^{1/4}
  \end{align*}
  for a universal constant $C$.
\end{proposition}

\begin{proof}
  Let $g = 1_A - \gamma_n(A)$ and $h = f - \E f$, so that $g$ and $h$
  both have mean zero and $\E[gh] = \Cov(1_A, f)$.
  Write $h = c g + h^\perp$, where $\E[g h^{\perp}] = 0$; then
  $c = \E[gh]/\E[g^2] = \Cov(1_A, f) / \Var(1_A)$. Since $P_t f - \E f = P_t h$,
  we have
  \begin{equation}\label{eq:variance-decomp}
    \Var(P_t f) = \E [(P_t h)^2] = \E [c^2 (P_t g)^2 + (P_t h^\perp)^2 + 2 c P_t g P_t h^\perp].
  \end{equation}
  Now, $\E [(P_t g)^2] = \Var(P_t 1_A)$ and $\E [(P_t h^\perp)^2] \ge 0$. For the last term,
  since $\E [g h^\perp] = 0$, the Cauchy-Schwarz inequality implies
  \[
    \E[P_t g P_t h^\perp] = \E[h^\perp P_{2t} g] = \E [h^\perp (P_{2t} g - g)] \ge
    - \sqrt{\E[(h^\perp)^2] \E [(P_{2t} g - g)^2]}.
  \]
  Since $P_{2t} g - g = P_{2t} 1_A - 1_A$, Lemma~\ref{lem:half-space-l2} implies that
  \[
    \E[P_t h P_t g] \ge -\frac{\sqrt{\E[(h^\perp)^2] \arccos(e^{-2t})}}{\sqrt \pi}.
  \]
  Going back to~\eqref{eq:variance-decomp} and using the bound $\E[(h^\perp)^2] \le \E[h^2] \le 1$,
  \[
    \Var(P_t f) \ge c^2 \Var(P_t 1_A) - \frac{\sqrt{\arccos(e^{-2t})}}{\sqrt\pi}.
  \]
  Recalling that $c = \Cov(A, f) / \Var(1_A)$, this proves the first claimed inequality.

  For the second inequality, note that Lemma~\ref{lem:half-space-l2} implies that
  \[
    \frac{\Var(P_t 1_A)}{\Var(A)} \ge 1 - \frac{\arccos(e^{-2t})}{2\pi \Var(A)}.
  \]
  Combining this with the first claimed inequality,
  \begin{align*}
    \Var(P_t f)
    &\ge \frac{\Cov(A, f)^2}{\Var(A)} \left(1 - \frac{C \arccos(e^{-2t})}{\Var(A)}\right) - C t^{1/4} \\
    &\ge 4 \Cov(A, f)^2 - C \frac{\Cov(A, f)^2}{\Var(A)} \arccos(e^{-2t}) - C t^{1/4}.
  \end{align*}
  Finally, $\Cov(A, f)^2 \le \Var(A)$ and $\arccos(e^{-2t}) \le C t^{1/4}$, thus proving the second inequality.
\end{proof}

In order to relate the noise stability of $f$ to half-spaces correlated with $f_{t,y}$, note that
\[
  \E_Y \E[f_{s,Y} P_{2t} f_{s,Y}] = \E [f P_{2r} f]
\]
when $e^{-2r} = e^{-2s} + e^{-2t} - e^{-2s-2t}$. Hence,
\[
  \Var(P_r f) = \E_Y \Var(P_t f_{s,Y}) + \Var(P_s f).
\]
Now, the Poincar\'e inequality implies that $\Var(P_s f) \le e^{-2(s-r)} \Var(P_r f)$; hence,
\[
  (1 - e^{-2(s-r)}) \Var(P_r f) \ge \E_Y \Var(P_t f_{s,Y}).
\]
By Proposition~\ref{prop:correlation-to-stability} applied to $f_{s,Y}$,
\[
  (1 - e^{-2(s-r)}) \Var(P_r f) \ge 4 \E_Y M^2(f_{s,Y}) - C t^{1/4}.
\]
To prove Theorem~\ref{thm:gaussian-converse}, note that if we fix $r$ and $s$ and solve for $t$
the we obtain $e^{-2t} = 1 - \frac{1 - e^{-2r}}{1 - e^{-2s}}$. For small $t$, this gives
$t = \Theta(\frac{1 - e^{-2r}}{1 - e^{-2s}})$ (while for large $t$ the Theorem is vacuous anyway).

\section{Boolean functions}

For this section, $P_t$ denotes the Bonami-Beckner semigroup defined in
Section~\ref{sec:intro-boolean}. Recall also the definition of
$f_z$ for $z \in \{-1, 0, 1\}^n$ from that section.
Let $\mu_t$ be the probability distribution
$e^{-t} \delta_0 + \frac 12 (1 - e^{-t}) (\delta_1 + \delta_{-1})$ on $\{-1, 0, 1\}$
and take $Z_t \sim \mu_t^{\otimes n}$.
Then we have the following relationship between $P_t$ and $Z_t$:
\[
    (P_t f)(x) = \E f_{Z_t}(x).
\]

\begin{theorem}\label{thm:boolean-restriction}
  For any $f: \{-1, 1\}^n \to [0, 1]$ and any $t > 0$,
  \[
      \E M(f_{Z_s}) \ge c (e^{2t} - 1) \Var(P_t f),
  \]
  where $s = - \log (1-e^{-t})$, $Z_s \sim \mu_s$ and $c > 0$ is a universal constant.
\end{theorem}

Before proceeding with the proof of Theorem~\ref{thm:boolean-restriction}, let
us make some remarks about how sharp it is. First of all, it is no longer
true if we replace $f_{Z_t}$ by $f$; that is, noise stable functions are not
necessarily correlated with half-spaces. We demonstrate this using a boolean
version of the earlier Gaussian example; details are in
Section~\ref{sec:boolean-example}.

Next, Theorem~\ref{thm:boolean-restriction} has a qualitative converse, which we will
state later as Theorem~\ref{thm:boolean-converse}. That is, if $M(f_{Z_s})$ is non-negligible
on average then $f$ is noise stable. In particular, Theorem~\ref{thm:boolean-restriction}
and Theorem~\ref{thm:boolean-converse} imply Theorem~\ref{thm:boolean-intro}.

Finally, Theorem~\ref{thm:boolean-restriction} implies that $M(f_{Z_t}) \ge c' (e^{2t} - 1) \Var(P_t f)$
with constant probability over $Z_t$. It turns out that this probability estimate cannot be
substantially improved. As an example, consider the function
\[
f(x) = \begin{cases}
  x_2 & \text{if $x_1 = 1$} \\
  \prod_{i=3}^n x_i &\text{if $x_1 = -1$}.
\end{cases}
\]
Then $f$ is noise-stable, but if $z_1 = -1$ then $f_z$ is noise sensitive and uncorrelated with
any half-space. In other words, $f_{Z_t}$ has probability
$\frac 12 e^{-t}$ of failing to be correlated with any half-space.

\subsection{Proof of Theorem~\ref{thm:boolean-restriction}}

The proof of Theorem~\ref{thm:boolean-restriction} follows the same lines
as the proof of Theorem~\ref{thm:gaussian}, but it requires a little background
on Fourier analysis of boolean functions: for a set $S \subset \{1, \dots, n\}$,
define $\chi_S: \{-1, 1\}^n \to \{-1, 1\}$ by
\[
    \chi_S(x) = \prod_{i \in S} x_i.
\]
It is well-known (see e.g.~\cite{ODonnell:14}) that $\{\chi_S: S \subset \{1, \dots, n\}\}$
is an orthonormal basis of $L_2(\{-1, 1\}^n)$; in particular, every $f: \{-1, 1\}^n \to [0, 1]$
may be expanded in this basis: define $\hat f(S)$ as the coefficients of this expansion:
\[
    f(x) = \sum_{S \subset \{-1, 1\}^n} \hat f(S) \chi_S(x).
\]
Also, we abbreviate $\hat f(\{i\})$ by $\hat f(i)$, and we define
\[
    w_1(f) = \sum_{i=1}^n \hat f(i)^2.
\]
We will show that if $w_1(f)$ is non-negligible then there is
a half-space correlated with $f$. Then we will show that $\E w_1(f_{Z_t})$ is
non-negligible.

\begin{proposition}\label{prop:weight-half-space}
If $w_1(f) = \eps^2$ and $\Var(f) = \sigma^2$ then there exists a half-space $B$ with 
$\Cov(f,1_B) \geq c \frac{\epsilon^2}{\sigma}$, where $c > 0$ is a universal constant.
\end{proposition}

The proof of Proposition~\ref{prop:weight-half-space} require two preparatory lemmas.
First, we observe
that it suffices to find a half-space which is correlated with $P_t f$ for some
$t > 0$:

\begin{lemma}\label{lem:smoothed-cov-to-cov}
If there exists a half-space $B$ with $\Cov(P_t f,1_B) \geq \delta$ then there exists a
half-space $\Cov(f,1_{B'}) \geq \delta$. 
\end{lemma}

\begin{proof}
  Suppose that $B = \{x: \sum_{i=1}^n a_i x_i \le b\}$. Take $X$ and $Y$
  to be independent, uniform random variables in $\{-1, 1\}^n$ and let
  $I \subset \{1, \dots, n\}$ be the random set that includes each element independently
  with probability $e^{-t}$. If $B(I,y)$ denotes the set
  $\{x: \sum_{i \in I} a_i x_i \le b - \sum_{i \not \in I} a_i y_i\}$ then
  \begin{align*}
  P_t 1_B (x)
  &= \Pr\left(\sum_{i \in I} a_i x_i + \sum_{i \not \in I} a_i Y_i \le b\right) \\
  &= \E 1_{B(I,Y)}(x).
  \end{align*}
  Since $P_t$ is self-adjoint,
  \[
  \Cov(P_t f, 1_B) = \E[(f-\E f) P_t 1_B] = \E [(f(X) - \E f) 1_{B(I,Y)}(X)].
  \]
  If the right hand side is larger than $\delta$ then in particular there exist $I$ and $y$
  such that
  \[
  \Cov(f, 1_{B(I,y)}) = \E [(f - \E f) 1_{B(I,y)}] \ge \delta.
  \qedhere
  \]
\end{proof} 

Next, we consider the case of linear functions. Up to constant factors, the best
possible correlation
between a linear function and a half-space is determined by the $L_2$ norm of the
function's coefficients. This is the first point where the boolean proof diverges
from the Gaussian proof: the Gaussian case of Lemma~\ref{lem:halfspace-linear-correlation}
is trivial (with a better constant) because of the Gaussian measure's rotational invariance.

\begin{lemma}\label{lem:halfspace-linear-correlation}
  If $\ell(x) = \sum a_i x_i$ and $B = \{x: \ell(x) \ge 0\}$
  then
  \[\E [\ell(X) 1_B(X)] \ge \|a\|_2/40.\]
\end{lemma}

\begin{proof}
  Since $\ell$ has mean zero,
  \[
  \E[\ell(X) 1_B(X)]
  = \frac 12 \E[\ell(X) (2 1_B(X) - 1)]
  = \frac 12 \E |\ell(X)|.
  \]
  Now, for any $M \ge 0$
  \begin{align}
  \E |\ell(X)|
  &\ge \E [|\ell(X)| 1_{\{|\ell(X)| \le M\}}] \notag \\
  &\ge \frac 1M \E [\ell^2(X) 1_{\{|\ell(x)| \le M\}}] \notag \\
  &= \frac 1M \left(\E [\ell^2(X)] - \E [\ell^2(X) 1_{\{|\ell(X)| > M\}}] \right).
  \label{eq:ell1-ell2}
  \end{align}
  Hoeffding's inequality implies that $\Pr(|\ell(X)| > t \|a\|_2) \le 2 e^{-t^2/2}$;
  hence,
  \begin{align*}
  \E [\ell^2(X) 1_{\{|\ell(X)| > M\}}]
  &= M^2 \Pr(\ell^2(X) \ge M^2) + \int_{M^2}^\infty \Pr(\ell^2(X) \ge s)\, ds \\
  &\le 2 M^2 e^{-M^2/(2\|a\|_2^2)} + 2 \int_{M^2}^\infty e^{-s/(2\|a\|_2^2)}\, ds \\
  &= 4 M^2 e^{-M^2/(2\|a\|_2^2)}.
  \end{align*}
  Setting $M = 10 \|a\|_2$, we have
  \[
    \E [\ell^2(X) 1_{\{|\ell(X)| > M\}}]
    \le 400 \|a\|_2^2 e^{-50} \le \frac 12 \|a\|_2^2.
  \]
  On the other hand, $\E[\ell^2(X)] = \|a\|_2^2$; going back
  to~\eqref{eq:ell1-ell2}, we have
  \[
  \E |\ell(X)| \ge \frac{1}{10 \|a\|_2} \left(\|a\|_2^2 - \frac 12 \|a\|_2^2\right)
  = \frac{\|a\|_2}{20}. \qedhere
  \]
\end{proof}

\begin{proof}[Proof of Proposition~\ref{prop:weight-half-space}]
 Write $f = \E f + f_1 + f_2$ where
 $f_1(x) = \sum_i x_i \hat f(i)$, and
  $f_2$ is orthogonal to both $f_1$ and 1.
  Note that $P_t f_1 = e^{-t} f_1$, while
  $\E (P_t f_2)^2 \le e^{-4t} \E f_2^2$. Hence,
  \begin{align*}
    \E [P_t f 1_B] - \E f \E 1_B
    &= \E [1_B P_t f_1] + \E [1_B P_t f_2] \\
    &\ge e^{-t} \E[1_B f_1] - e^{-2t} \|1_B\|_2 \|f_2\|_2 \\
    &\ge e^{-t} \E[1_B f_1] - e^{-2t} \sigma.
  \end{align*}
  Now, Lemma~\ref{lem:halfspace-linear-correlation} implies that
  there exists a half-space $B$ with $\E[1_B f_1] \ge \epsilon/40$.
  For this $B$,
  \[
  \E [P_t f 1_B] - \E f \E 1_B \ge \frac{e^{-t}}{40} \epsilon - e^{-2t} \sigma.
  \]
  If we take $t$ to solve $e^{-t} = \epsilon / (80 \sigma)$ then
  \[
  \E [P_t f 1_B] - \E f \E 1_B \ge c \frac{\epsilon^2}{\sigma}
  \]
  for a universal constant $c > 0$.
  By Lemma~\ref{lem:smoothed-cov-to-cov}, there exists
  some half-space $B'$ with $\Cov(f, 1_{B'}) \ge c \frac{\epsilon^2}{\sigma}$.
\end{proof}  

Next, we show that $\E[w_1(f_Z)]$ is substantial if $f$ is noise-stable.
\begin{proposition}\label{prop:restriction-weight}
  For any $t > 0$, if $e^{-s} = 1 - e^{-t}$ then
\[
\E[w_1(f_{Z_s})] \geq (1-e^{-t})  \sum_{S} |S| \hat{f}^2(S) e^{-2t(|S|-1)} \geq 
(e^{2t} - e^t)\Var(P_t f).
\]
\end{proposition} 

\begin{proof}
Fix $t$ and set $Z = Z_s$. Recalling the definition of $w_1$, we have
\[
\E[w_1(f_{Z})] = \sum_{i=1}^n \E[\hat f_{Z}^2(i))].
\]
Note that $\hat f_{Z}(i) = 0$ if $Z_i = \pm 1$, which happens with probability $1-e^{-t}$.
Otherwise $\hat f_{Z}(i)$ is given by 
\begin{equation}\label{eq:hat_f_Z}
  \hat f_{Z}(i) = \sum_{S : i \in S} \hat{f}(S) \prod_{j \in S \setminus \{i\}} Z_j.
\end{equation}
Therefore
\begin{align*}
\E[\hat f_Z(i)^2]
&= (1-e^{-t}) \sum_{S,T : i \in S, i \in T} \hat{f}(S) \hat{f}(T) \E[ \prod_{j \in S \setminus \{i\}} Z_j \prod_{k \in T \setminus \{i\}} Z_k] \\
&= (1-e^{-t}) \sum_{S : i \in S} \hat{f}^2(S) e^{-2t(|S|-1)}.
\end{align*}
Summing over $i$ proves the first inequality; the second follows from the fact that
\[
\Var(P_t f)
= \sum_{|S| \ge 1} e^{-2t|S|} \hat f^2(S)
\le \sum_{S} |S| e^{-2t|S|} \hat f^2(S).
\qedhere
\]
\end{proof} 

\begin{proof}[Proof of Theorem~\ref{thm:boolean-restriction}]
  Take $s$ so that $e^{-s} = 1 - e^{-t}$ and apply Proposition~\ref{prop:restriction-weight}:
  $\E w_1 (f_{Z_s}) \ge (e^{2t} - e^t) \Var(P_t f)$. By Proposition~\ref{prop:weight-half-space}
  and because $\Var(f_{Z_s}) \le 1$,
  \[
      \E M(f_{Z_s}) \ge c \E w_1(f_{Z_s}) \ge c (e^{2t} - e^t) \Var(P_t f).
  \]
  Finally, $e^{2t} - e^t = e^t(e^t - 1)  \ge \frac 12 (e^t + 1) (e^t - 1) = \frac 12 (e^{2t} - 1)$.
\end{proof}

\subsection{An example}\label{sec:boolean-example}

Let $n = m^2$, and let $J_i = \{(i-1)m, \dots, im - 1\}$.
Let $B_n \subset \{-1, 1\}^n$ be the set
\[
\left\{
  x: \sum_{i=1}^m \left(\frac{1}{\sqrt m} \sum_{j \in J_i} x_j \right)^2 \le m
\right\}.
\]

From the central limit theorem, one sees immediately that $B_n$ is noise stable, with
the same estimate as its Gaussian analogue in Section~\ref{sec:gaussian-example}.
\begin{proposition}\label{prop:example-boolean-stable}
  For any $n$ and any $t > 0$,
  \[
    \Var(P_t 1_{B_n}) \ge \frac 14 - \frac{\arccos(e^{-2t})}{\sqrt 2 \pi} - o_n(1).
  \]
  In particular $B_n$ is noise stable.
\end{proposition}

Finally, we show that $B_n$ is not correlated with any half-space. This
essentially follows from the invariance principle, which says that nice boolean
functions have almost the same distribution when their arguments are replaced
by Gaussian variables.

\begin{proposition}\label{prop:boolean-example}
  $M(B_n) \le C m^{-1/200}$
\end{proposition}

For the rest of this section, fix $x \in \R^n$ and $b \in \R$, and suppose
that $A = \{x \in \{-1, 1\}^n: \sum_i a_i x_i \le b\}$.
Let $J^* \subset \{1, \dots, n\}$ be
the set containing the indices of the $\lfloor m^{1/3} \rfloor$ largest $|a_i|$.
Define
$a^+$ by $a^+_i = 1_{\{i \in J^*\}} a_i$ and set $a^- = a - a^+$.

We split our proof of Proposition~\ref{prop:boolean-example} into two parts,
depending on the decay properties of $a$.
If $a^-$ is unbalanced,
it follows that $a^+$ must contain only large coordinates. We apply the
Littlewood-Offord theorem to argue that $a^-$ is essentially irrelevant and $A$
depends only on a few coordinates. Since $B_n$ doesn't depend on any small set of
coordinates, this implies that $A$ and $B_n$ are uncorrelated.
If $a^-$ is fairly balanced then we condition on $\{X_i: i \in J^*\}$ and
apply an invariance principle to $\{X_i: i \not \in J^*\}$, replacing
boolean variables with Gaussian variables and applying
Proposition~\ref{prop:example-gaussian-no-half-space}.

First, we recall the Littlewood-Offord inequality:
\begin{theorem}\label{thm:littlewood-offord}
  If $X$ is uniformly distributed in $\{-1, 1\}^n$ then
  \[
      \sup_{c \in \R} \Pr\left(\big|\sum_i X_i a_i - c\big| \le t \min_i |a_i|\right)
      \le C t n^{-1/2}.
  \]
\end{theorem}

\begin{lemma}\label{lem:lo-case}
  If $\|a^-\|_\infty \ge m^{-1/24} \|a^-\|_2$ then
  $\Cov(1_A, 1_{B_n}) \le C m^{-1/12}$.
\end{lemma}

\begin{proof}
By Theorem~\ref{thm:littlewood-offord} and since $|a_i| \ge \|a^-\|_\infty$ for
all $i \in J^*$,
\[
    \Pr\left(
      \Big|\sum_{j \in J^*} a_j X_j - b \Big| \le m^{1/24} \|a^-\|_2
    \right) \le C m^{1/12} |J^*|^{-1/2} \le C m^{-1/12}.
\]
On the other hand, Chebyshev's inequality implies that
\[
    \Pr\left(
      \Big|\sum_{j \not \in J^*} a_j X_j \Big| \ge m^{1/24} \|a^-\|_2
    \right) \le m^{-1/12}.
\]
Putting these two inequalities together, we see that with probability at least
$1 - C m^{-1/12}$ over $\{X_j: j \in J^*\}$ we have
\begin{equation}\label{eq:A-determined}
    \Pr(X \in A \mid X_j: j \in J^*) \in [0, m^{-1/12}] \cup [1-m^{-1/12}, 1].
\end{equation}
On the other hand, conditioning on $\{X_j: j \in J^*\}$ has little effect on the
event $B_n$: each random variable $Z_i := \big(\sum_{j \in J_i} X_j\big)^2$ has conditional
expectation $m \pm O(|J_i \cap J^*|^2)$ and conditional variance $O(m)$; moreover,
$\E [|Z_i - \E Z_i|^3] = O(m^{3/2})$. Then
\[
    \sum_{i=1}^m \Big(\sum_{j \in J_i} X_j\Big)^2
\]
has conditional expectation $m^2 \pm O(|J^*|^2) = m^2 \pm O(m^{2/3})$. By the Berry-Esseen theorem,
\[
    \Pr(X \in B_n \mid X_j: j \in J^*) = \frac 12 \pm O(m^{-1/2}).
\]
Combined with~\eqref{eq:A-determined}, this implies that
\[
    \E [(1_{B_n} - \Pr(B_n)) 1_A \mid X_j: j \in J^*] \le C m^{-1/12}
\]
with probability at least $1 - C m^{-1/12}$. Integrating over $\{X_j: j \in J^*\}$,
this implies the claim.
\end{proof}

Since Lemma~\ref{lem:lo-case} implies Proposition~\ref{prop:boolean-example} in the
case $\|a^-\|_\infty \ge m^{-1/24} \|a^-\|_2$, we may assume from now on
that $\|a^-\|_\infty \le m^{-1/24} \|a^-\|_2$.
We will prove the remaining case of Proposition~\ref{prop:boolean-example} in two steps:
for the rest of the section, let $X$ be uniform on $\{-1, 1\}^n$ and take $Y \sim \gamma_n$;
note that $A$ and $B_n$ can be canonically extended to subsets of $\R^n$.

For any $c \in \R$, let $h_c: \R \to [0, 1]$ be the function $h_c(x) = 1_{\{x \le c\}}$.
For $\epsilon > 0$, let $h_{c,\epsilon}$ be a function satisfying
\begin{itemize}
  \item $h_{c,\epsilon}$ takes values in $[0, 1]$,
  \item $h_{c,\epsilon}(x) = h_c(x)$ for all $x$ such that $|x - c| \ge \epsilon$, and
  \item for $k = 1, 2, 3$, $h_{c,\epsilon}^{(k)}$ is uniformly bounded by
    $C \epsilon^{-k}$ for some universal constant $C$ (where $h^{(k)}$ denotes the $k$th derivative of $h$).
\end{itemize}
For $z \in \{-1, 1\}^{J^*}$ and let
$\Omega_z$ be the event $\{X_i = z_i\ \forall i \in J^*\}$. 
Set $J'_i = J_i \setminus J^*$ and $s_i = \sum_{j \in J_i \cap J^*} z_i$.
Next, define the polynomials
\begin{align*}
  p(x) &= \frac{1}{m^2} \sum_i \Big(\sum_{j \in J_i} x_j\Big)^2 \\
  p_z(x) &= \frac{1}{m^2} \sum_i \Big(\sum_{j \in J'_i} x_j + s_i\Big)^2 \\
  q_z(x) &= \frac{1}{\|a^-\|} \Big(\sum_{j \not \in J^*} a_j x_j + \sum_{j \in J^*} a_j z_j\Big).
\end{align*}
Recalling (from the Berry-Esseen theorem) that $\Pr(X \in B_n) = \frac 12 + O(m^{-1/2})$, our
goal is to show that
\[
    \E \Big[1_A(X) \Big(1_{B_n}(X) - \frac 12\Big)\Big] \le C m^{-1/12}.
\]
We will achieve this by conditioning on $\Omega_z$: for an arbitrary $z$, we claim that
\[
    \E \Big[1_A(X) \Big(1_{B_n}(X) - \frac 12\Big)\ \Big\mid\ \Omega_z\Big] \le C m^{-1/12}.
\]
Going back to the definitions of $p_z$ and $q_z$, this is equivalent to
\begin{equation}\label{eq:invariance-goal}
    \E \Big[h_{b'}(q_z(X)) \Big(h_1(p_z(X)) - \frac 12\Big)\Big] \le C m^{-1/12},
\end{equation}

We divide the proof of~\eqref{eq:invariance-goal} into several steps: for any $\epsilon > 0$,
\begin{gather}
  \E |h_1(p_z(X)) - h_1(p(X))| \le C m^{-1/6} \label{eq:uncondition-p} \\
  \E |h_{1,\epsilon}(p(X)) - h_1(p(X))| \le C \max\{\epsilon, m^{-1/2}\} \label{eq:noise-p-X} \\
  \E |h_{b',\epsilon}(q_z(X)) - h_{b'}(q_z(X))| \le C \max\{\epsilon, m^{-1/24}\} \label{eq:noise-q-X} \\
  | \E [h_{b',\epsilon}(q_z(X)) h_{1,\epsilon}(p(X))] - \E [h_{b',\epsilon}(q_z(Y)) h_{1,\epsilon}(p(Y))] | \le C \epsilon^{-3} m^{-1/48} \label{eq:invariance} \\
  \E |h_{1,\epsilon}(p(Y)) - h_1(p(Y))| \le C \max\{\epsilon, m^{-1/2}\} \label{eq:noise-p-Y} \\
  \E |h_{b',\epsilon}(q_z(Y)) - h_{b'}(q_z(Y))| \le C \max\{\epsilon, m^{-1/2}\} \label{eq:noise-q-Y}  \\
  \Cov(h_{b'}(q_z(Y)), h_1(p(Y))) \le C m^{-1/2} \label{eq:cov-Y}.
\end{gather}
Taking $\epsilon = m^{-1/200}$ and
combining~\eqref{eq:uncondition-p} through~\eqref{eq:cov-Y} using the triangle inequality yields~\eqref{eq:invariance-goal}.

Fortunately, most of the pieces above are easy:~\eqref{eq:noise-p-X} follows from the Berry-Esseen
theorem, since $h_{1,\epsilon}$ and $h_1$ are both bounded by one, and agree except on an interval
of length $2\epsilon$. Inequalities~\eqref{eq:noise-q-X},~\eqref{eq:noise-p-Y}, and~\eqref{eq:noise-q-Y}
follow by the same argument (the reason for the worse bound in~\eqref{eq:noise-q-X} is because
the error term in the Berry-Esseen theorem depends on $\|a^-\|_\infty / \|a^-\|_2$, which we
only know to be bounded by $m^{-1/24}$).

It remains to check~\eqref{eq:uncondition-p},~\eqref{eq:invariance}, and~\eqref{eq:cov-Y}; for these,
it helps to introduce the notion of influences: for function $f: \{-1, 1\}^n \to \R$, we define
the influence of the $i$th coordinate to be
\[
  \Inf_i(f) = \Var \E[f(X) \mid X_1, \dots, X_{i-1}, X_{i+1}, \dots, X_n].
\]
If the range of $f$ is $\{-1, 1\}$ then $\Inf_i(f)$ is just the probability that negating $X_i$
will change the value of $f(X)$.

For~\eqref{eq:uncondition-p}, note that the Berry-Esseen theorem applied to the variables 
$S_k := \big(\sum_{j \in J_k} X_j\big)$ implies that with probability at least $1 - Cm^{-1/6}$,
$h_1(p(X))$ falls outside the interval $[1 - 6 m^{-2/3}, 1 + 6 m^{-2/3}]$. Hence, in order to change
the value of $h_1(p(X))$, one would need to change the value of $\sum_k S_k^2$ by at least $6 m^{4/3}$.
On the other hand, Hoeffding's inequality implies that with probability at least $1 - Cm^{-1/6}$,
$\max_k |S_k| \le 2 m$.
On this event, in order to change the value of $\sum_k S_k^2$ by $6 m^{4/3}$,
one would need to change at least $2 m^{1/3}$ of the $X_j$. Since $p_z(X)$ is obtained from $p(X)$
by changing at most $m^{1/3}$ of the $X_j$, we see that $h_1(p(X)) = h_1(p_z(X))$ unless one of the two
events above fails. This proves~\eqref{eq:uncondition-p}.

Recognizing that $h_1(p(Y)) = 1_{B_n}(Y)$ and $h_{b'}(q_z(Y))$ is the indicator function
of some half-space, the following Lemma proves~\eqref{eq:cov-Y}.
\begin{lemma}\label{lem:cov-Y}
  For any half-space $A$, $\Cov(1_A(Y), 1_{B_n}(Y)) \le m^{-1/2}$
\end{lemma}
\begin{proof}
The covariance in question can be written in terms of covariances between
half-spaces and $m$-dimensional balls, which we may then bound
using Proposition~\ref{prop:example-gaussian-no-half-space}.
To do this, we break each block of $m$
variables in terms of its contribution in the $(1,\ldots,1)$ direction and the
contribution in the orthogonal direction: for each block $J$ of $m$ variables, define 
\[
  x_J = m^{-1/2} \sum_{j \in J} x_j, \quad 
  a_J = \E\big[X_J \sum_{j \in J} a_j x_j\big] =  m^{-1/2} \sum_{j \in J} a_j
\]
and 
\[
  r_J = \sqrt{\sum_{j \in J} a_j^2 - a_J^2}, \quad r = \sqrt{ \sum_{J} r_J^2}. 
\]
Now define $A', B' \subset \R^{m+1}$ by
\begin{align*}
    A' &= \left\{x \in \R^{m+1}: \sum_{i=1}^m a_{J_i} x_i + r x_{m+1} \le b\right\} \\
    B' &= \left\{x \in \R^{m+1}: \sum_{i=1}^m x_i^2 \le m\right\}.
\end{align*}
Note that $A'$ and $B'$ are the push-forwards of $\tilde A_n$ and $\tilde B$ under a map
that preserves the standard Gaussian measure: if $\Pi_m: \R^n \to \R^m$ is defined by
$\Pi_m x = (x_{J_1}, \dots, x_{J_m})$ and $\Pi$ is defined by
\[
    \Pi x = (\Pi_m x, r^{-1}(\inr{a}{x} - \inr{\Pi_m a}{\Pi_m x})
\]
then $x \in A$ (resp. $B$) if and only if $\Pi x \in A'$ (resp. $B'$).
Since $\Pi$ pushes forward $\gamma_n$ onto $\gamma_{m+1}$, we have
\[
\Cov(1_{\tilde{A}},1_{\tilde{B_n}}) = \Cov(1_{A'},1_{B'}).
\]
On the other hand, $\Cov(1_{A'}, 1_{B'}) \le m^{-1/2}$ by
Proposition~\ref{prop:example-gaussian-no-half-space}.
\end{proof}

Finally,~\eqref{eq:invariance} follows from the following multivariate invariance principle
that was proved by the first author in~\cite{Mossel:10}:

\begin{theorem}\label{thm:invariance}
    Suppose $p(x)$ and $q(x)$ are polynomials of degree at most $d$ such that $\Inf_i(p) \le \tau$
    and $\Inf_i(q) \le \tau$ for all $i$. For any $\Psi: \R^2 \to \R$ with third partial derivatives
    uniformly bounded by $B$,
    \[
      |\E \Psi(p(X), q(X)) - \E \Psi(p(Y), q(Y))|
      \le C^d d B \sqrt \tau,
    \]
    where $Y \sim \gamma_n$, $X$ is uniform on $\{-1, 1\}^n$, and $C$ is a universal constant.
\end{theorem}

Taking $d = 2$, $\tau = m^{-1/24}$ and $\Psi(x, y) = h_{1,\epsilon}(x) h_{b',\epsilon}(y)$
(which has third derivatives bounded by $C \epsilon^{-3}$) proves~\eqref{eq:invariance}.

\subsection{The converse of Theorem~\ref{thm:boolean-restriction}}

Here, we state and prove the boolean analogue of Theorem~\ref{thm:gaussian-converse}
(or, the qualitative converse of Theorem~\ref{thm:boolean-restriction}).
That is, we show that if $M(f_{s,Y})$ is non-negligible with constant probability
then $f$ is noise stable.

\begin{theorem}\label{thm:boolean-converse}
  For any $0 < r < s$ and any $f: \{-1, 1\}^n \to [-1, 1]$,
  \[
    (1 - e^{-2(s-r)}) \Var(P_r f) \ge 4 \E M^2(f_{Z_s}) - C \left(\frac{1-e^{-2r}}{1-e^{-2s}}\right)^{1/4},
  \]
  where $Z_s \sim \mu_s$ and $C$ is a universal constant.
\end{theorem}

The proof of Theorem~\ref{thm:boolean-converse} is very much like the proof of
Theorem~\ref{thm:gaussian-converse}, so we give only a sketch. As in the proof
of Theorem~\ref{thm:gaussian-converse}, the first step is a bound on the noise
stability of half-spaces. However, the bound that we used to prove Lemma~\ref{lem:half-space-l2}
is equivalent to an open question (the ``majority is least stable conjecture'') in
the boolean case, so we use a weaker (by a constant factor) bound due to Peres~\cite{Peres:04}:

\begin{theorem}\label{thm:peres}
  For any half-space $A$ and any $t > 0$, $\E [(1_A - P_t 1_A)^2] \le C \sqrt t$,
  where $C$ is a universal constant.
\end{theorem}

Next, we show that any set which is correlated with a half-space must be noise stable (indeed,
almost as noise stable as the half-space itself).

\begin{proposition}\label{prop:correlation-to-stability-boolean}
  Suppose that $A \subset \{-1, 1\}^n$ is a half-space. Then for any $f: \{-1, 1\}^n \to [0, 1]$
  and any $t > 0$,
  \[
    \Var(P_t f) \ge 4 \Cov(A, f)^2 - C t^{1/4}
  \]
  for a universal constant $C$.
\end{proposition}

The proof of Proposition~\ref{prop:correlation-to-stability-boolean} is
essentially identical to the proof of Proposition~\ref{prop:correlation-to-stability},
so we omit it. The only difference is that we use Theorem~\ref{thm:peres} instead
of Lemma~\ref{lem:half-space-l2}.

Finally, the argument to go from Proposition~\ref{prop:correlation-to-stability-boolean}
to Theorem~\ref{thm:boolean-converse} is also essentially identical to the Gaussian case:
the only property of Gaussians that we used in that argument was the Poincar\'e inequality,
which takes the same form in the boolean case.

\subsection{Acknowledgement}
We thank Dana Moshkovitz, Gil Kalai and Irit Dinur for encouragement to complete this work. 
E.M acknowledges the support of NSF grant CCF 1320105, DOD ONR grant N00014-14-1-0823, and grant 328025 from the Simons Foundation".

\bibliographystyle{plain}
\bibliography{all,my}
\end{document}